\documentclass[12pt]{amsart}
\usepackage{cite}
\usepackage{amsthm}
\usepackage{amssymb}
\usepackage{longtable}
\usepackage{geometry}
\usepackage{enumerate}
\usepackage{setspace}
\usepackage{upgreek}
\usepackage{tikz}
\usepackage{tikz-cd}
\usetikzlibrary{matrix, arrows}
\usepackage[unicode=true]
 {hyperref}
\hypersetup{
colorlinks=true,
urlcolor=black,
citecolor=blue,
linkcolor=blue,
}

\newtheorem{thm}{Theorem}[section]
\newtheorem{prop}[thm]{Proposition}
\newtheorem{lem}[thm]{Lemma}

\theoremstyle{definition}

\newtheorem*{assumption}{Assumption}
\theoremstyle{remark}

\numberwithin{equation}{section}

\newcommand{\ff}{\mathfrak{f}}
\newcommand{\fg}{\mathfrak{g}}
\newcommand{\fh}{\mathfrak{h}}
\newcommand{\bZ}{\mathbb{Z}}
\newcommand{\bN}{\mathbb{N}}
\newcommand{\Perm}{\mathrm{Perm}}
\newcommand{\Hol}{\mathrm{Hol}}
\newcommand{\Aut}{\mathrm{Aut}}
\newcommand{\Inn}{\mathrm{Inn}}
\newcommand{\Out}{\mathrm{Out}}
\newcommand{\Hom}{\mathrm{Hom}}
\newcommand{\conj}{\mathrm{conj}}
\newcommand{\GL}{\mathrm{GL}}
\newcommand{\SL}{\mathrm{SL}}
\newcommand{\PSL}{\mathrm{PSL}}
\newcommand{\PSU}{\mathrm{PSU}}

\begin{document}

\large 

\title[Hopf-Galois structures on finite extensions with quasisimple Galois group]{Hopf-Galois structures on finite extensions \\with quasisimple Galois group}

\author{Cindy (Sin Yi) Tsang}
\address{School of Mathematics (Zhuhai), Sun Yat-Sen University, Zhuhai, Guangdong, China}
\email{zengshy26@mail.sysu.edu.cn}\urladdr{http://sites.google.com/site/cindysinyitsang/} 

\date{\today}

\maketitle

\vspace{-2mm}

\begin{abstract}Let $L/K$ be a finite Galois extension of fields with Galois group $G$. It is known that $L/K$ admits exactly two Hopf-Galois structures when $G$ is non-abelian simple. In this paper, we extend this result to the case when $G$ is quasisimple. \end{abstract}

\vspace{-1mm}

\tableofcontents

\vspace{-4mm}

\section{Introduction}

Let $L/K$ be a finite Galois extension of fields with Galois group $G$. By \cite{GP}, we know that each Hopf-Galois structure $\mathcal{H}$ on $L/K$ is associated to a group $N_{\mathcal{H}}$ of the same order as $G$. For each group $N$ of order $|G|$, define
\[ e(G,N) = \#\{\mbox{Hopf-Galois structures $\mathcal{H}$ on $L/K$ with $N_\mathcal{H}\simeq N$}\}.\]
Let $\Perm(N)$ be the group of all permutations on $N$. Recall that a subgroup of $\Perm(N)$ is \emph{regular} if its action on $N$ is regular. For example, clearly $\lambda(N)$ and $\rho(N)$ are regular subgroups of $\Perm(N)$, where
\[\begin{cases}
\lambda:N\longrightarrow \Perm(N);\hspace{1em}\lambda(\eta) = (x\mapsto\eta x)\\
\rho:N\longrightarrow\Perm(N);\hspace{1em}\rho(\eta) = (x\mapsto x\eta^{-1})
\end{cases}\]
are the left and right regular representations of $N$. By work of \cite{GP} and \cite{By96}, we have the formula
\[e(G,N) = \frac{|\Aut(G)|}{|\Aut(N)|}\cdot\#\left\{\begin{array}{c}\mbox{regular subgroups of $\Hol(N)$}\\\mbox{which are isomorphic to $G$}\end{array}\right\},\]
where $\Hol(N)$ is the \emph{holomorph} of $N$ and is defined to be
\[ \Hol(N) = \rho(N) \rtimes \Aut(N).\]
The calculation of $e(G,N)$ has been an active line of research because Hopf-Galois structures have application in Galois module theory; see \cite{Childs book} for more details. Let us also note in passing that regular subgroups of the holomorph are related to set-theoretic solutions to the Yang-Baxter equation; see \cite{YBE}.

\vspace{1mm}

For $N\simeq G$, the number $e(G,N)$ must be non-zero because $\lambda(N)$ and $\rho(N)$ are regular subgroups of $\Hol(N)$; note that $\lambda(N)$ and $\rho(N)$ are equal exactly when $N$ is abelian. For $N\not\simeq G$, the number $e(G,N)$ could very well be zero. In certain extreme cases, it might happen that
\begin{equation}\label{extreme} e(G,N) = \begin{cases} 1 & \mbox{for $N\simeq G$ when $G$ is abelian},\\
2 & \mbox{for $N\simeq G$ when $G$ is non-abelian},\\
0 &\mbox{for all other $N\not\simeq G$}.
\end{cases}\end{equation}
For $G$ abelian, by \cite[Theorem 1]{By96} we know exactly when (\ref{extreme}) occurs:

\begin{thm}If $G$ is a finite abelian group, then $(\ref{extreme})$ holds precisely when the orders of $G$ and $(\bZ/|G|\bZ)^\times$ are coprime.
\end{thm}

For $G$ non-abelian, the situation is more complicated. By \cite{Childs simple,Byott simple}, we have:

\begin{thm}\label{simple thm}If $G$ is a finite non-abelian simple group, then $(\ref{extreme})$ holds.
\end{thm}

It is natural to ask whether Theorem~\ref{simple thm} may be generalized to other non-abelian groups $G$ which are close to being simple. Recall that $G$ is said to be \emph{quasisimple} if $G=[G,G]$ and $G/Z(G)$ is a simple group, where $[G,G]$ is the commutator subgroup and $Z(G)$ is the center of $G$. In \cite[Theorem 1.3]{Tsang HG}, the author has already shown that:

\begin{thm}\label{QS thm}If $G$ is a finite quasisimple group, then $e(G,G)=2$.
\end{thm}

It remains to consider the groups $N\not\simeq G$. In \cite[Theorem 1.6]{Tsang HG}, the author has shown that if $G$ is the double cover of $A_n$ with $n\geq 5$, then $e(G,N)=0$ for all groups $N\not\simeq G$ of order $n!$. We shall extend this result and prove:

\begin{thm}\label{thm}If $G$ is a finite quasisimple group, then $e(G,N)=0$ for all groups $N\not\simeq G$ of order $|G|$.
\end{thm}

In view of Theorems~\ref{QS thm} and~\ref{thm}, one might guess that (\ref{extreme}) is also true for say,  finite almost simple or non-abelian characteristically simple groups $G$. If $G=S_n$ with $n\geq 5$, however, then by \cite[Theorems 5 and 9]{Childs simple}, we have
\[ e(G,G)\neq 2\mbox{ and }e(G,N)\neq0\mbox{ for some }N\not\simeq G.\]
See \cite{Tsang AS1,Tsang AS} for generalizations to other finite almost simple groups $G$. If $G$ is a finite non-abelian characteristically simple group which is not simple, then $e(G,G)\neq 2$ by \cite{Tsang char simple}, but as far as the author knows, there is no investigation yet on whether there exists $N\not\simeq G$ such that $e(G,N)\neq0$.

\section{Crossed homomorphisms}

In this section, let $G$ and $\Gamma$ be finite groups, whose orders are not assumed to be equal. Given $\ff\in\Hom(G,\Aut(\Gamma))$, recall that a \emph{crossed homomorphism} (with respect to $\ff$) is a map $\fg:G\longrightarrow\Gamma$ which satisfies
\[ \fg(\sigma\tau) = \fg(\sigma)\cdot \ff(\sigma)(\fg(\tau))\mbox{ for all }\sigma,\tau\in G.\]
Let $Z_\ff^1(G,\Gamma)$ be the set of all such maps $\fg$. The regular subgroups of $\Hol(\Gamma)$ isomorphic to $G$ may be parametrized by the bijective maps in $Z_\ff^1(G,\Gamma)$.

\begin{prop}\label{fg prop}The regular subgroups of $\Hol(\Gamma)$ isomorphic to $G$ are precisely the sets $\{\rho(\fg(\sigma))\cdot\ff(\sigma):\sigma\in G\}$, as $\ff$ ranges over $\Hom(G,\Aut(\Gamma))$ and $\fg$ over the bijective maps in $Z_\ff^1(G,\Gamma)$.
\end{prop}
\begin{proof}The proof is straightforward; see \cite[Proposition 2.1]{Tsang HG}.\end{proof}

Hence, when $\Gamma$ has order $|G|$, that $e(G,\Gamma)$ is non-zero is equivalent to the existence of a bijective map $\fg\in Z_\ff^1(G,\Gamma)$ for some $\ff\in\Hom(G,\Aut(\Gamma))$. Let us give two approaches to study these crossed homomorphisms. The first is to define another $\fh\in\Hom(G,\Aut(\Gamma))$. The idea originates from \cite{Childs simple} and was formalized by the author in \cite[Proposition 3.4]{Tsang AS1} or \cite[Proposition 2.3]{Tsang AS}. The second is to use \emph{characteristic} subgroups of $\Gamma$, that is, subgroups $\Lambda$ such that $\varphi(\Lambda) = \Lambda$ for all $\varphi\in \Aut(\Gamma)$. %Note that in this case we have a natural homomorphism
%\begin{equation}\label{char hom} \Aut(\Gamma)\longrightarrow \Aut(\Gamma/\Lambda);\hspace{1em}\varphi\mapsto(\gamma\Lambda\mapsto\varphi(\gamma)\Lambda).\end{equation}
%The next two results are from \cite[Proposition 2.3]{Tsang AS} and \cite[Lemma 4.1]{Tsang HG}; their proofs are straightforward.
The idea comes from \cite{Byott simple} and was restated in terms of crossed homomorphisms by the author in \cite[Lemma 4.1]{Tsang HG}. 

\begin{prop}\label{h prop}Let $\ff\in\Hom(G,\Aut(\Gamma))$ and $\fg\in Z_\ff^1(G,\Gamma)$. Define
\[ \fh : G\longrightarrow \Aut(\Gamma);\hspace{1em} \fh(\sigma) = \conj(\fg(\sigma))\cdot\ff(\sigma),\]
where $\conj(\cdot)=\lambda(\cdot)\rho(\cdot)$. Then:
\begin{enumerate}[(a)]
\item We have $\fh\in\Hom(G,\Aut(\Gamma))$.
\item For any $\sigma\in G$, we have $\ff(\sigma) = \fh(\sigma)$ if and only if $\sigma\in \fg^{-1}(Z(\Gamma))$.
\item The map $\sigma\mapsto\fg(\sigma)$ defines a homomorphism $\ker(\ff)\longrightarrow \Gamma$.
\item The map $\sigma\mapsto\fg(\sigma)^{-1}$ defines a homomorphism $\ker(\fh)\longrightarrow\Gamma$.
\end{enumerate}
\end{prop}
\begin{proof}Part (a) appeared in \cite[Proposition 3.4]{Tsang AS1}, while parts (b) -- (d) were stated in \cite[Proposition 2.3]{Tsang AS}. The proofs are straightforward.
\end{proof}

\begin{prop}\label{char prop}Let $\ff\in\Hom(G,\Aut(\Gamma))$ and $\fg\in Z_\ff^1(G,\Gamma)$. Let $\Lambda$ be any characteristic subgroup of $\Gamma$. Consider the homomorphism
\[\begin{tikzcd}[column sep=3cm]
\overline{\ff}:G\arrow{r}{\ff} & \Aut(\Gamma) \arrow{r}{\varphi\mapsto(\gamma\Lambda\mapsto\varphi(\gamma)\Lambda)} & \Aut(\Gamma/\Lambda)
\end{tikzcd}\]
induced by $\ff$ and the map
\[\begin{tikzcd}[column sep=3cm]
\overline{\fg}:G\arrow{r}{\fg} & \Gamma \arrow{r}{\text{\tiny quotient map}} & \Gamma/\Lambda
\end{tikzcd}\]
induced by $\fg$. Then:
\begin{enumerate}[(a)]
\item We have $\overline{\fg}\in Z_{\overline{\ff}}^1(G,\Gamma/\Lambda)$.
\item The preimage $\fg^{-1}(\Lambda)$ is a subgroup of $G$.
\item In the case that $\fg$ is bijective, there is a regular subgroup of $\Hol(\Lambda)$ which is isomorphic to $\fg^{-1}(\Lambda)$.
\end{enumerate}
\end{prop}
\begin{proof}
Parts (a) and (b) are clear; see \cite[Proposition 4.1]{Tsang HG} for a proof of the latter. For part (c), see \cite[Proposition 3.3]{Tsang solvable}.
\end{proof}

Following \cite{Byott simple}, we shall take $\Lambda$ to be a maximal characteristically subgroup of $\Gamma$. Then, the quotient $\Gamma/\Lambda$ is a non-trivial characteristically simple group, and since $\Gamma$ is finite, we know that
\begin{equation}\label{N/M} 
\Gamma/\Lambda\simeq T^m,\mbox{ where $T$ is a simple group and }m\in\bN,
\end{equation}
in which case the structure of $\Aut(\Gamma/\Lambda)$ is well-known. This approach turns out to be very useful and was crucial in all of \cite{Tsang HG,Tsang Sn, Tsang AS, Tsang solvable}.

\section{Consequences of CFSG}

In this section, let $A$ be a finite non-abelian simple group. We shall require some consequences of the classification of finite simple groups (CFSG).

\vspace{1mm}

One difficulty in dealing with finite quasisimple groups, as opposed to non-abelian simple groups, is that they need not be centerless. But their center is a quotient of the Schur multiplier of the associated finite non-abelian simple group; see \cite[Section 33]{A book} for more on Schur multipliers.

\vspace{1mm}

Let $\mathfrak{m}(A)$ denote the order of the Schur multiplier of $A$. We shall say that $\PSL_n(q)$ and $\PSU_n(q)$, respectively, are \emph{non-exceptional} if
\begin{equation}\label{non except} \mathfrak{m}(\PSL_n(q)) = \gcd(n,q-1)\mbox{ and }\mathfrak{m}(\PSU_n(q)) = \gcd(n,q+1).\end{equation}
As a consequence of CFSG, we know that $\mathfrak{m}(A)$ is small, in fact at most $12$, except when $A\simeq\PSL_n(q),\PSU_n(q)$. More specifically:

\begin{lem}\label{Schur multiplier lem} Let $\mathfrak{M}=\{1,2,3,4,6,12\}$. 
\begin{enumerate}[(a)]
\item If $A\not\simeq \PSL_n(q),\PSU_n(q)$, then $\mathfrak{m}(A)\in\mathfrak{M}$. 
\item If $A=\PSL_n(q),\PSU_n(q)$, then $\mathfrak{m}(A)\in\mathfrak{M}$ or $A$ is non-exceptional, except that $\mathfrak{m}(\PSL_3(4)) = 48$ and $\mathfrak{m}(\PSU_4(3))=36$.
\end{enumerate}
\end{lem}
\begin{proof}See \cite[Theorem 5.1.4]{groups book}.
\end{proof}

\begin{lem}\label{Out lem} The outer automorphism group $\Out(A)$ of $A$ is solvable.
\end{lem}
\begin{proof}This is known as Schreier conjecture; see \cite[Theorem 1.46]{G book}.
\end{proof}

\begin{lem}\label{fpf lem} Every $\varphi\in\Aut(A)$ has a fixed point other than $1_A$.
\end{lem}
\begin{proof}See \cite[Theorem 1.48]{G book}.\end{proof}

\begin{lem}\label{factorization}There do not exist subgroups $B_1$ and $B_2$ of $A$ such that both of them have non-trivial center and $A = B_1B_2$.
\end{lem}
\begin{proof}This was a conjecture of Szep and was proven in \cite{Szep conj}.
\end{proof}

\begin{lem}\label{prime power index}Suppose that $A$ has a subgroup of index $p^a$, where $p$ is a prime and $a\in\bN$. Then, one of the following holds:
\begin{enumerate}[(a)]
\item $A\simeq A_{p^a}$ with $p^a\geq 5$;
\item $A\simeq\PSL_n(q)$ with $p^a = (q^n-1)/(q-1)$;
\item $A\simeq\PSL_2(11)$ with $p^a=11$;
\item $A\simeq M_{11}$ with $p^a=11$, or $A\simeq M_{23}$ with $p^a=23$;
\item $A\simeq\PSU_4(2)$ with $p^a=27$.
\end{enumerate}
\end{lem}
\begin{proof}See \cite[Theorem 1]{RG}.
\end{proof}

\section{Proof of Theorem~$\ref{thm}$}

In this section, let $G$ be a finite quasisimple group, and let $N$ be any group of order $|G|$. Suppose that $e(G,N)\neq0$, namely there is a regular subgroup $\mathcal{G}$ of $\Hol(N)$ isomorphic to $G$. In the next two subsections, we shall prove:

\begin{prop}\label{thm1}The group $N$ must be perfect, namely $N=[N,N]$. \end{prop}

\begin{prop}\label{thm2}The regular subgroup $\mathcal{G}$ is either $\rho(N)$ or $\lambda(N)$.\end{prop}

Theorem \ref{thm} would follow, because then $N\simeq \mathcal{G}\simeq G$.

\vspace{1mm}

Let us first set up the notation. By Proposition~\ref{fg prop}, we know that
\[\mathcal{G}=\{\rho(\fg(\sigma))\cdot\ff(\sigma):\sigma\in G\},\mbox{ where }
\begin{cases}\ff\in\Hom(G,\Aut(N)),\\\fg\in Z_\ff^1(G,N)\mbox{ is bijective}.\end{cases}\]
Alternatively, we may rewrite it as
\[\mathcal{G}= \{\lambda(\fg(\sigma))^{-1}\cdot\fh(\sigma):\sigma\in G\},\mbox{ where }\fh\in\Hom(G,\Aut(N))\]
is defined as in Proposition~\ref{h prop}.  Let $M$ be any maximal characteristic subgroup of $N$. Then, as in Proposition~\ref{char prop}, we have homomorphisms
\[\begin{tikzcd}[column sep=3cm]
\overline{\ff},\overline{\fh}:G\arrow{r}{\ff,\fh} & \Aut(N) \arrow{r}{\varphi\mapsto(\eta M\mapsto\varphi(\eta)M)} & \Aut(N/M)
\end{tikzcd}\]
induced by $\ff$ and $\fh$, respectively, and a surjective crossed homomorphism 
\[\begin{tikzcd}[column sep=3cm]
\overline{\fg}:G\arrow{r}{\fg} & N \arrow{r}{\text{\tiny quotient map}} & N/M
\end{tikzcd}\]
with respect to $\overline{\ff}$ induced by $\fg$. We shall also need the facts that
\begin{align}\label{quotient of M}
&\mbox{$Z(G)$ is a quotient of the Schur multiplier of $G/Z(G)$},\\\label{QS norsub}
&\mbox{all proper normal subgroups of $G$ are contained in $Z(G)$}.
\end{align}
See \cite[(33.8)]{A book} for the former, and the latter is an easy exercise.

\subsection{Non-perfect groups}

Suppose for contradiction that $N$ is not perfect, in which case we may take $M$ to contain $[N,N]$. By (\ref{N/M}), we then have
\[ N/M \simeq (\bZ/p\bZ)^m,\mbox{ where $p$ is a prime and }m\in\bN.\]
Let us first make a simple observation.
\begin{lem}\label{m neq 1}The homomorphism $\overline{\ff}$ is non-trivial and $m\geq 2$.
\end{lem}
\begin{proof}Suppose for contradiction that $\overline{\ff}$ is trivial. Then, the map
\[\begin{tikzcd}[column sep = 2cm]
G \arrow{r}{\overline{\fg}} & N/M \arrow{r}{\simeq} &(\bZ/p\bZ)^m
\end{tikzcd}\]
is a homomorphism by Proposition~\ref{h prop}(c), and so must be trivial because $G$ is perfect. This contradicts that $\overline{\fg}$ is surjective, so indeed $\overline{\ff}$ is non-trivial. It follows that $m\geq2$, for otherwise
\[\begin{tikzcd}[column sep = 2cm]
G \arrow{r}{\overline{\ff}} & \Aut(N/M) \arrow{r}{\simeq} & (\bZ/p\bZ)^\times
\end{tikzcd}\]
would be trivial again because $G$ is perfect, which we know is impossible.
\end{proof}

Define $H = \fg^{-1}(M)$, which is a subgroup of $G$ by Proposition~\ref{char prop}(b), and whose order is $|M|$ because $\fg$ is bijective. Thus, we have $[G:H] = p^m$. Put
\[ p^z = [HZ(G):H] = [Z(G) : H\cap Z(G) ],\mbox{ where }0\leq z\leq m.\]
Note that $p$ divides $|Z(G)|$ unless $z=0$. Also $m-z\geq 1$, for otherwise
\[ G = HZ(G)\mbox{ and in particular }G= [G,G] = [H,H],\]
which is impossible because $H$ is a proper subgroup. Since
\[ \left[\frac{G}{Z(G)} : \frac{HZ(G)}{Z(G)}\right] = \frac{[G:H]}{[HZ(G):H]} = p^{m-z},\]
by Lemma~\ref{prime power index} one of the following holds:
\begin{enumerate}[(a)]
\item $G/Z(G)\simeq A_{p^{m-z}}$ with $p^{m-z}\geq 5$;
\item $G/Z(G)\simeq\PSL_n(q)$ with $p^{m-z} = (q^n-1)/(q-1)$;
\item $G/Z(G)\simeq\PSL_2(11)$ with $p^{m-z} = 11$;
\item $G/Z(G)\simeq M_{11}$ with $p^{m-z} = 11$, or $G/Z(G)\simeq M_{23}$ with $p^{m-z} = 23$;
\item $G/Z(G)\simeq\mbox{PSU}_4(2)$ with $p^{m-z} = 27$.
\end{enumerate}
Since Theorem~\ref{thm} holds when $G$ is a finite non-abelian simple group by \cite{Byott simple}, and when $G$ is the double cover of $A_n$ with $n\geq 5$ by \cite[Theorem 1.6]{Tsang HG}, we may henceforth assume that:

\begin{assumption}The center $Z(G)$ of $G$ is non-trivial.\end{assumption}

Recall from (\ref{quotient of M}) that $|Z(G)|$ divides $\mathfrak{m}(G/Z(G))$. Hence, this assumption in particular restricts that $\mathfrak{m}(G/Z(G))\neq1$.

\begin{assumption}The group $G$ is not the double cover of an alternating group.\end{assumption}

\begin{lem}\label{case b lem}We must be in case $(b)$.
\end{lem}
\begin{proof}Case (d) does not occur by our first assumption because
\[ \mathfrak{m}(M_{11}) = 1 = \mathfrak{m}(M_{23}).\]
To deal with cases (a), (c), and (e), note that
\[\mathfrak{m}(A_{n}) = \begin{cases}2&\mbox{if $n=5$ or $n\geq 8$}\\
 6 & \mbox{if $n=6,7$}\end{cases}\mbox{ and }\mathfrak{m}(\PSL_2(11)) = 2 = \mathfrak{m}(\PSU_4(2)).\]
For case (a), we must have $p^{m-z}=7$ by our second assumption. For case (c), we have $p^{m-z}=11$. In both cases, note that $p$ does not divide $\mathfrak{m}(G/Z(G))$, so $z=0$ and $m=1$. But this contradicts Lemma~\ref{m neq 1}. For case (e), we have $p^{m-z}=27$. Since $p$ does not divide $\mathfrak{m}(G/Z(G))$, we have $z=0$ and $m=3$. Also, by our first assumption, necessarily
\[ |Z(G)| = 2,\mbox{ and so }|G| = 2|\PSU_4(2)| = 51840.\]
But $|\GL_m(p)| = |\GL_3(3)|= 11232$, so the homomorphism
\[\begin{tikzcd}[column sep = 2cm]
G \arrow{r}{\overline{\ff}} & \Aut(N/M) \arrow{r}{\simeq} &\GL_3(3)
\end{tikzcd}\]
is trivial by (\ref{QS norsub}) and by comparing orders. But this contradicts Lemma~\ref{m neq 1}. Thus, indeed we must be in case (b).
\end{proof}

In view of Lemma~\ref{case b lem}, we now know that
\[G/Z(G)\simeq \PSL_n(q)\mbox{ with }p^{m-z} = (q^n-1)/(q-1).\]
By \cite[Theorem 5.1.4]{groups book}, we also know that $\mathfrak{m}(\PSL_n(q)) = \gcd(n,q-1)$, unless $(n,q)$ equals one of the five pairs stated in the next lemma. Let us first rule out these cases.

\begin{lem}\label{exceptional lem}We have $(n,q)\neq(2,4),(2,9),(3,2),(3,4),(4,2)$.
\end{lem}
\begin{proof}Suppose for contradiction that $(n,q)$ is one of the stated pairs. Then this pair must be $(2,4)$ or $(3,2)$ because $(q^n-1)/(q-1)$ is a prime power. Note that
\[ p^{m-z} = \frac{q^n-1}{q-1} = \begin{cases} 
5&\mbox{if }(n,q) = (2,4),\\
7&\mbox{if }(n,q) = (3,2).
\end{cases}\]
But $\mathfrak{m}(\PSL_2(4)) = 2 = \mathfrak{m}(\PSL_3(2))$. In either case, since $p$ does not divide $2$, we see that $z=0$, and so $m=1$. But this contradicts Lemma~\ref{m neq 1}.
\end{proof}

\begin{lem}\label{SL lem}We have $G\simeq \SL_n(q)$ and $|Z(G)|=n=p$.
\end{lem}
\begin{proof}By Lemma~\ref{exceptional lem} and our first assumption, respectively, we have
\[ \mathfrak{m}(G/Z(G)) = \gcd(n,q-1) \mbox{ and }\mathfrak{m}(G/Z(G)) \neq 1.\]
As noted in \cite[(3.3)]{RG}, that $p^{m-z}=(q^n-1)/(q-1)$ implies that $n$ is a prime. It then follows that
\[\gcd(n,q-1) = n,\mbox{ and so }q\equiv1\hspace{-3mm}\pmod{n}.\]
Moreover, we must have $|Z(G)|=n$ and also $G\simeq \SL_n(q)$, the universal cover of $\PSL_n(q)$. Note that
\[ p^{m-z} \equiv q^{n-1}+\cdots +q+1 \equiv 1+\cdots + 1 + 1 \equiv n \equiv 0 \hspace{-2mm}\pmod{n}\]
so in fact $n=p$. This completes the proof.
\end{proof}

We shall now use the next proposition to get a contradiction and thus prove Proposition \ref{thm1}; cf. \cite[Theorem 4.3]{Byott simple} and the argument in \cite[Section 4]{Byott simple}.

\begin{prop}\label{SL prop}Suppose that $\SL_n(q)$ has a non-trivial irreducible representation of degree $d$ over a field of characteristic coprime to $q$, where
\[ (n,q)\neq (3,2),(3,4),(4,2),(4,3).\]
Then, we have
\[ d\geq \begin{cases}(q-1)/\gcd(2,q-1) &\mbox{if $n=2$},\\ (q^n-1)/(q-1) - 2&\mbox{if $n\geq 3$}.\end{cases}\]
\end{prop}
\begin{proof}See \cite[Lemma 9.1.1 and Theorem 9.1.5]{SL bound}.
\end{proof}

\begin{proof}[Proof of Proposition~$\ref{thm1}$]
By Lemmas~\ref{m neq 1} and \ref{SL lem}, we have $m\geq 2$ and there is a non-trivial homomorphism
\[\begin{tikzcd}[column sep = 2cm]
\SL_n(q)  \arrow{r}{\simeq}  & G \arrow{r}{\overline{\ff}} & \Aut(N/M) \arrow{r}{\simeq} & \GL_m(p).
\end{tikzcd}\]
Since $p^{m-z}=(q^n-1)/(q-1)$, we have $\gcd(p,q)=1$ and also $(n,q)\neq (4,3)$. By Lemma~\ref{exceptional lem}, we may then apply Proposition~\ref{SL prop} as follows. Recall that $|Z(G)| = n = p$ by Lemma~\ref{SL lem}, which in turn implies $z=0,1$.
\vspace{1mm}

For $n=2$, note that $p^{m-z} = q+1$, and we obtain
\[ m \geq \frac{q-1}{\gcd(2,q-1)}= \frac{p^{m-z}-2}{\gcd(2,p^{m-z}-2)} \geq \frac{2^{m-1}-2}{2}.\]
For $n\geq 3$, similarly we have
\[ m \geq \frac{q^n-1}{q-1} - 2 = p^{m-z} - 2 \geq 3^{m-1}-2.\]
From the above inequalities, we deduce that
\[ (m,p) = \begin{cases}
(2,2),(3,2),(4,2) & \mbox{if }n=2,\\ (2,3) & \mbox{if }n\geq 3.
\end{cases}\]
Since $n^{m-z} = p^{m-z} = (q^n-1)/(q-1)$ and $z=0,1$, it follows that
\[ (n,q) = (2,3),(2,7),\mbox{ and in fact necessarily }(n,q) = (2,7)\]
because $\PSL_2(3)$ is not simple. We are left with the possibility $G\simeq \mathrm{SL}_2(7)$.

\vspace{1mm}

Using the {\tt Holomorph} and {\tt RegularSubgroups} commands  in \textsc{Magma} \cite{magma}, we checked that $\Hol(N)$ has no regular subgroup isomorphic to $G\simeq\SL_2(7)$ for all non-perfect groups $N$ of order $336$. We remark that in fact it suffices to check the insolvable groups $N$ of order $336$ by \cite[Theorem 1.10]{Tsang solvable}. Thus, we obtain a contradiction, so indeed $N$ must be perfect.\end{proof}

\subsection{Perfect groups}

By Proposition \ref{thm1}, we know that $N$ must be perfect, in which case all quotients of $N$ are also perfect. By (\ref{N/M}), we then have
\[ N/M \simeq T^m,\mbox{ where $T$ is non-abelian simple and }m\in\bN,\]
and by \cite[Lemma 3.2]{Byott simple} for example, we know that
\[ \Aut(T^m) = \Aut(T)^m \rtimes S_m.\]
%Here ``wr'' denotes wreath product and the action of $S_m$ is the canonical one. 
We shall also use Burnside's theorem, which states that the order of a finite insolvable group is divisible by at least three distinct primes.

\vspace{1mm}

The cases $G/Z(G)\simeq \PSL_n(q),\PSU_n(q)$ require special arguments because then $\mathfrak{m}(G/Z(G))$ and so $|Z(G)|$ could be arbitrarily large. Let us recall that
\begin{align}\label{order}
|\PSL_n(q)| & = \frac{1}{\gcd(n,q-1)} \left(q^{{n\choose 2}}\prod_{i=2}^{n}(q^i-1)\right),\\\notag
|\PSU_n(q)| & = \frac{1}{\gcd(n,q+1)}\left(q^{{n\choose 2}}\prod_{i=2}^{n}(q^i-(-1)^i)\right).
\end{align}
We shall prove that either $\ff$ or $\fh$ is trivial in a sequence of steps.

\begin{lem}\label{Inn lem}The image $\overline{\ff}(G)$ lies in $\Inn(N/M)$.
\end{lem}
\begin{proof} Below, we shall show that the homomorphism
\[ \begin{tikzcd}[column sep = 2cm]
\overline{\ff}_{S_m}:G \arrow{r}{\overline{\ff}} & \Aut(N/M) \ar[equal]{r}{ \mbox{\tiny identification}}  & \Aut(T)^m \rtimes S_m \arrow{r}{\text{\tiny projection}} & S_m
\end{tikzcd}\]
is trivial. Then, the image $\overline{\ff}(G)$ lies in $\Aut(T)^m$, and the homomorphism
\[ \begin{tikzcd}[column sep = 2cm]
G \arrow{r}{\overline{\ff}} & \Aut(T)^m \arrow{r}{\text{\tiny projection}} & \Out(T)^m
\end{tikzcd}\]
is also trivial, because $G$ is perfect while $\Out(T)$ is solvable by Lemma~\ref{Out lem}. \par\noindent It follows that $\overline{\ff}(G)$ lies in $\Inn(T)^m$, which is identified with $\Inn(N/M)$.

\vspace{1mm}

To prove that $\overline{\ff}_{S_m}$ is trivial, let $\ell$ be any prime factor of $|T|$. For any finite group $\Gamma$, let $v_\ell(\Gamma)$ be the non-negative integer such that $\ell^{v_\ell(\Gamma)}$ exactly divides $|\Gamma|$. We have $v_\ell(G) \geq m$ because $|G| = |N| = |T|^m|M|$. It is well-known that
\[ v_\ell(S_m) = \left\lfloor \frac{m}{\ell}\right\rfloor + \left\lfloor \frac{m}{\ell^2}\right\rfloor + \left\lfloor \frac{m}{\ell^3}\right\rfloor +\cdots\mbox{ and so }v_\ell(S_m)< \frac{m}{\ell-1}.\]
Since $G/\ker(\overline{\ff}_{S_m})$ embeds into $S_m$, we then deduce that
\begin{equation}\label{embed ineq} v_\ell(G) - v_\ell(\ker(\overline{\ff}_{S_m})) \leq v_\ell(S_m) < \frac{m}{\ell-1}.\end{equation}
Suppose now for contradiction that $\overline{\ff}_{S_m}$ is non-trivial, in which case $\ker(\overline{\ff}_{S_m})$ lies in $Z(G)$ by (\ref{QS norsub}). From $m\leq v_\ell(G)$ and (\ref{embed ineq}), we see that
\[ m - v_\ell(Z(G)) \leq v_\ell(G) - v_\ell(Z(G)) \leq v_\ell(G) - v_\ell(\ker(\overline{\ff}_{S_m}))  < \frac{m}{\ell-1},\]
and so $v_\ell(Z(G))\geq1$. This implies that every prime factor of $|T|$ also divides $|Z(G)|$. From Burnside's theorem and (\ref{quotient of M}), it then follows that $\mathfrak{m}(G/Z(G))$ has at least three distinct prime divisors. From Lemma~\ref{Schur multiplier lem}, we deduce that
\[ G/Z(G)\simeq \PSL_n(q),\PSU_n(q)\mbox{ with $\PSL_n(q),\PSU_n(q)$ non-exceptional}.\]
Put $v_\ell(G)=x$ and $v_\ell(Z(G))=y$, where $x,y\geq1$. Then, we have
\[x - y < \frac{m}{\ell-1} \leq\frac{x}{\ell-1} \mbox{ and in particular }y > \frac{\ell-2}{\ell-1} \cdot x.\]
Also, from (\ref{non except}) and (\ref{quotient of M}), we see that
\begin{equation}\label{divisible}|Z(G)|\mbox{ divides }\begin{cases}\gcd(n,q-1) & \mbox{if }G/Z(G)\simeq\PSL_n(q),\\
\gcd(n,q+1)& \mbox{if }G/Z(G)\simeq \PSU_n(q).\end{cases}\end{equation}
Since $\ell^y$ divides $|Z(G)|$, from (\ref{divisible}) we have $\ell^y\leq n$, that is $y\leq \log(n)/\log(\ell)$. Observe that the order formulae in (\ref{order}) imply that
\begin{align*}
 (q-1)^{n-2} &\mbox{ divides }|\PSL_n(q)|,\\
 (q+1)^{\lfloor n/2\rfloor -1} &\mbox{ divides }|\PSU_n(q)|.\end{align*}
Since $\ell^y$ divides $|Z(G)|$, again from (\ref{divisible}) we see that $\ell$ divides $q-1$ and $q+1$, respectively, and in particular
\[ v_\ell(G) - v_\ell(Z(G)) \geq \begin{cases}
n-2 & \mbox{if }G/Z(G)\simeq\PSL_n(q),\\
\left\lfloor\frac{n}{2}\right\rfloor - 1& \mbox{if }G/Z(G)\simeq \PSU_n(q).
\end{cases}\]
In either case, this in turns yields
\[ x - 1 \geq x - y  \geq  \frac{n}{2}-\frac{1}{2} - 1\mbox{ and so }x\geq \frac{n-1}{2}.\]
Again, by Burnside's theorem, we may take $\ell\geq 5$. We obtain
\[ \frac{4}{3}\cdot\frac{\log(n)}{\log(5)} \geq \frac{\ell-1}{\ell-2}\cdot\frac{\log(n)}{\log(\ell)}\geq\frac{\ell-1}{\ell-2}\cdot y > x\geq \frac{n-1}{2}.\]
But then $n=2$, which contradicts that $5\leq \ell^y\leq n$. Hence, indeed $\overline{\ff}_{S_m}$ must be trivial, and this completes the proof.
\end{proof}

\begin{lem}\label{m=1}We have $N/M\simeq T$.
\end{lem}
\begin{proof}We have $\Inn(N/M)\simeq N/M \simeq T^m$. Depending on whether $\overline{\ff}$ is trivial or not, by Proposition~\ref{h prop}(c) and Lemma~\ref{Inn lem}, respectively, we see that there is a non-trivial homomorphism $\varphi:G\longrightarrow T^m$. Let $1\leq i\leq m$ be such that
\[ \begin{tikzcd}[column sep = 2cm]
\varphi^{(i)}:G \arrow{r}{\varphi} & T^m \arrow{r}{\text{\tiny projection}} & T^{(i)}\mbox{ (the $i$th copy of $T$)}
\end{tikzcd}\]
is non-trivial. Since $\ker(\varphi^{(i)})$ lies in $Z(G)$ by (\ref{QS norsub}), we have
\[  \frac{|T|^m|M|}{|Z(G)|}[Z(G):\ker(\varphi^{(i)})] = \frac{|G|}{|\ker(\varphi^{(i)})|} = |\varphi^{(i)}(G)| = \frac{|T|}{[T^{(i)}:\varphi^{(i)}(G)]},\]
and in particular
\begin{equation}\label{Z(G) eqn} |Z(G)| = |T|^{m-1}|M| [Z(G):\ker(\varphi^{(i)})] [T^{(i)}:\varphi^{(i)}(G)].\end{equation}
Suppose for contradiction that $m\geq 2$, in which case $|T|$ divides $|Z(G)|$ and hence $\mathfrak{m}(G/Z(G))$ by (\ref{quotient of M}). Since all groups of order at most $48$ are solvable, from Lemma~\ref{Schur multiplier lem} and (\ref{non except}), we see that
\[ G/Z(G)\simeq \PSL_n(\ell^a),\PSU_n(\ell^a),\mbox{ where $\ell$ is a prime, and $\ell\nmid |Z(G)|$}.\]
But $\ell$ divides $|G| = |T|^m|M|$ by (\ref{order}) and thus $|T||M|$. This shows that (\ref{Z(G) eqn}) cannot hold, so indeed $m=1$, and we have $N/M\simeq T$.
\end{proof}

For any $\sigma\in G$, recall that $\overline{\fh}(\sigma) = \conj(\overline{\fg}(\sigma))\cdot\overline{\ff}(\sigma)$ by definition, and so 
\begin{equation}\label{fixed points} \overline{\ff}(\sigma) = \overline{\fh}(\sigma)\iff \overline{\fg}(\sigma) = 1_{N/M} \iff \sigma\in  \fg^{-1}(M)\end{equation}
because $N/M$ has trivial center.

\begin{lem}\label{M = Z(G)}We have $G/Z(G)\simeq T$ and $|M| = |Z(G)|$.
\end{lem}
\begin{proof}By Lemma~\ref{Inn lem}, the image $\overline{\ff}(G)$ lies in $\Inn(N/M)$, and so plainly $\overline{\fh}(G)$ lies in $\Inn(N/M)$ as well. Since $\Inn(N/M)\simeq N/M$, we then have
\[ |G/\ker(\overline{\ff})|, |G/\ker(\overline{\fh})|\leq |N/M|,\mbox{ and so } |M| \leq |\ker(\overline{\ff})|,|\ker(\overline{\fh})|.\]
Trivially $\overline{\ff}(\sigma) = \overline{\fh}(\sigma)$ for all $\sigma\in\ker(\overline{\ff})\cap\ker(\overline{\fh})$, so by (\ref{fixed points}) we have
\[ |\ker(\overline{\ff})\cap \ker(\overline{\fh})| \leq |\fg^{-1}(M)|,\mbox{ and }|\fg^{-1}(M)| = |M|\]
because $\fg$ is bijective. Since $\overline{\fg}$ is surjective, we also have the factorization
\[ \Inn(N/M) = \overline{\ff}(G)\overline{\fh}(G),\mbox{ whence $\overline{\ff}(G)$ or $\overline{\fh}(G)$ has trivial center}\]
by Lemmas~\ref{factorization} and~\ref{m=1}. This implies that $\ker(\overline{\ff})\subset\ker(\overline{\fh})$ or $\ker(\overline{\fh})\subset\ker(\overline{\ff})$ has to hold, for otherwise $\ker(\overline{\ff}),\ker(\overline{\fh})\lneq Z(G)$ by (\ref{QS norsub}), and both $\overline{\ff}(G)$ and $\overline{\fh}(G)$ would have non-trivial center. By symmetry, we may assume that the former inclusion holds. Then, from the above inequalities, we deduce that
\[ |M| = |\ker(\overline{\ff})|,\mbox{ and so }G/\ker(\overline{\ff})\simeq \Inn(N/M) \simeq T\]
by comparing orders. But $\ker(\overline{\ff})$ lies in $Z(G)$ again by (\ref{QS norsub}), and $T$ has trivial center, so in fact $\ker(\overline{\ff})=Z(G)$. Both claims now follow.
\end{proof}

\begin{lem}\label{f or h trivial}Either $\overline{\ff}$ or $\overline{\fh}$ is trivial, and $Z(G) = \fg^{-1}(M)$.
\end{lem}
\begin{proof}Suppose for contradiction that both $\overline{\ff}$ and $\overline{\fh}$ are non-trivial. By (\ref{QS norsub}), this means that both $\ker(\overline{\ff})$ and $\ker(\overline{\fh})$ lie in $Z(G)$. Since $G/\ker(\overline{\ff}),G/\ker(\overline{\fh})$ embed into $\Inn(N/M)$, by comparing orders and by Lemma~\ref{M = Z(G)}, we have
\[ \ker(\overline{\ff}) = Z(G) = \ker(\overline{\fh})\mbox{ and } \overline{\ff}(G) = \Inn(N/M) = \overline{\fh}(G).\]
From (\ref{fixed points}), we then deduce that $Z(G) \subset \fg^{-1}(M)$, which must be an equality by Lemma~\ref{M = Z(G)} and the bijectivity of $\fg$. The above also implies that $\overline{\ff}$ and $\overline{\fh}$, respectively, induce isomorphisms
\[ \varphi_f,\varphi_h: G/Z(G)\longrightarrow \Inn(N/M),\mbox{ and }\varphi_h^{-1}\circ\varphi_f\in\Aut(G/Z(G)).\]
But for any $\sigma\in G$, again by (\ref{fixed points}), we have
\[ (\varphi_h^{-1}\circ\varphi_f)(\sigma Z(G)) = \sigma Z(G) \iff \sigma \in \fg^{-1}(M) \iff \sigma Z(G) = 1_{G/Z(G)}.\]
This contradicts Lemma~\ref{fpf lem}. Thus, at least one of $\overline{\ff}$ or $\overline{\fh}$ is trivial.

\vspace{1mm}

Now, by Proposition~\ref{h prop}(c),(d), the surjective map
\[\varphi:G\longrightarrow N/M;\hspace{1em}\varphi(\sigma) = \begin{cases}
 \overline{\fg}(\sigma) &\mbox{if $\overline{\ff}$ is trivial}\\ \overline{\fg}(\sigma)^{-1} &\mbox{if $\overline{\fh}$ is trivial}
\end{cases}\]
is a homomorphism, and $\ker(\varphi)$ lies in $Z(G)$ by (\ref{QS norsub}). By comparing orders, we see from Lemma~\ref{M = Z(G)} that in fact $\ker(\varphi) = Z(G)$. But in both cases, we have $\ker(\varphi) = \fg^{-1}(M)$ by definition, so the claim follows.
\end{proof}

\begin{lem}\label{M central}We have $M\subset Z(N)$.
\end{lem}
\begin{proof}Since $M$ is normal in $N$, we have a homomorphism
\[ \begin{tikzcd}[column sep = 3cm]
\Phi:N \arrow{r}{\eta\mapsto (x\mapsto\eta x\eta^{-1})} & \Aut(M) \arrow{r} \arrow{r}{\text{\tiny projection}} & \Out(M)
\end{tikzcd}\]
whose kernel clearly contains $M$. Either $\Phi$ is trivial or $\ker(\Phi)=M$ because $N/M$ is simple by Lemma~\ref{m=1}.

\vspace{1mm}

Suppose first that $\Phi$ is trivial. This implies that
\[ N = MC,\mbox{ where $C$ is the centralizer of $M$ in $N$}.\]
Given $i\in\bN_{\geq 0}$ and a group $\Gamma$, let $\Gamma^{(i)}$ denote its $i$th derived subgroup. Since elements in $M$ and $C$ commute, we easily see that
\[ N^{(i)} = M^{(i)}C^{(i)} \mbox{ for all }i\in\bN_{\geq 0}.\]
By Lemma~\ref{f or h trivial} and Proposition~\ref{char prop}(c), there is a regular subgroup of $\Hol(M)$ which is isomorphic to $Z(G)$. Since $Z(G)$ is abelian, it then follows from \cite[Theorem 1.3(b)]{Tsang solvable} that $M$ is metabelian, namely $M^{(2)}=1$. Since $N$ is perfect, we deduce that
\[ N = N^{(1)} = N^{(2)} = M^{(2)}C^{(2)} = C^{(2)}\mbox{ and so }N = C.\]
This means that all elements in $N$ centralize $M$, that is $M\subset Z(N)$.

\vspace{1mm}

Suppose now that $\ker(\Phi)=M$, in which case $N/M$ embeds into $\Out(M)$. From Lemmas~\ref{m=1} and~\ref{M = Z(G)}, we then see that
\[ G/Z(G) \mbox{ embeds into } \Out(M),\mbox{ and recall }|M|=|Z(G)|.\]
Notice that then $\Out(M)$ and in particular $\Aut(M)$ must be insolvable. Let $\mathfrak{M} =\{1,2,3,4,6,12\}$ be as in Lemma~\ref{Schur multiplier lem}. We consider three cases.
\begin{enumerate}[1.]
\item $\mathfrak{m}(G/Z(G))$ lies in $\mathfrak{M}$: By (\ref{quotient of M}) we know that $|M|=|Z(G)|$ divides one of the numbers in $\mathfrak{M}$. But we checked in {\textsc{Magma}} \cite{magma} that no such group $M$ has insolvable $\Aut(M)$.
\item $G/Z(G)\simeq\PSL_3(4),\PSU_4(3)$: Recall Lemma~\ref{Schur multiplier lem}. Again by (\ref{quotient of M}) we know that $|M|=|Z(G)|$ divides $48$ or $36$. Since $\Aut(M)$ must be insolvable, we checked in {\textsc{Magma}} that $M$ has {\textsc{SmallGroup}} ID equal to one of
\begin{equation}\label{ID list} (8,5),(16,14),(24,15),(48,50),(48,51),(48,52),\end{equation}
and in particular $\mathfrak{m}(G/Z(G))$ is divisible by $8$. Hence, we have
\[G/Z(G) \simeq \PSL_3(4),\mbox{ and note that }|\PSL_3(4)| = 20160.\]
Again, using the {\tt OuterOrder} command, we computed in {\textsc{Magma}} that
\[ |\Out(M)| = 168,20160,336,120,1344,40320,\]
respectively, when $M$ has {\textsc{SmallGroup}} ID in (\ref{ID list}). Moreover, the group $M$ is abelian and there is no subgroup isomorphic to $\PSL_3(4)$ in $\Aut(M)$, when $M$ has {\textsc{SmallGroup}} ID equal to $(16,14),(48,52)$. We then deduce that $G/Z(G)$ cannot embed into $\Out(M)$.
\item $G/Z(G)\simeq\PSL_n(q),\PSU_n(q)$ with $\PSL_n(q),\PSU_n(q)$ non-exceptional: We may assume that $M\neq1$. Then, by \cite[Corollary 3.3]{Isaacs}, we have
\[ |\varphi|\leq |M|-1\mbox{ for all $\varphi\in\Aut(M)$}.\]
Since $|M| = |Z(G)|$, from (\ref{non except}) and (\ref{quotient of M}), we deduce that
\[ |\varphi\Inn(M)|\leq |\varphi| \leq \min\{n-1,q\} \mbox{ for all }\varphi\in\Aut(M).\]
Since $\Aut(M)$ is insolvable, we must have $n\geq 4$, and so $n=2+1+n_0$ for some integer $n_0\geq 1$. But then $G/Z(G)$ would contain an element of order $q^2 - 1 > q$ by \cite[Corollary 3(3)]{PSLPSU} and so cannot embed into $\Out(M)$.
\end{enumerate} 
In all cases, we obtained a contradiction. Hence, the case $\ker(\Phi)=M$ in fact does not occur, so indeed $M\subset Z(N)$.
\end{proof}

We are now ready to prove Proposition \ref{thm2}.

\begin{proof}[Proof of Proposition $\ref{thm2}$] By Lemma~\ref{f or h trivial}, either $\overline{\ff}$ or $\overline{\fh}$ is trivial.  Since $N$ is perfect, and $M\subset Z(N)$ by Lemma~\ref{M central}, the homomorphism
\[ \Aut(N)\longrightarrow \Aut(N/M);\hspace{1em}\varphi\mapsto (\eta M\mapsto \varphi(\eta)M)\]
is injective; see the proof of \cite[Proposition 3.5(c)]{Tsang HG}, for example. Therefore, either $\ff$ or $\fh$ is trivial. But clearly
\[ \mathcal{G} = \begin{cases}
\rho(N) &\mbox{if $\ff$ is trivial},\\ \lambda(N) &\mbox{if $\fh$ is trivial}.
\end{cases}\]
This completes the proof.
\end{proof}

\section*{Acknowledgments}

Research supported by the Young Scientists Fund of the National Natural Science Foundation of China (Award No.: 11901587).

\vspace{1mm}

The author thanks the referee for helpful comments.

\end{document}